\documentclass[reqno]{amsart}
\usepackage[colorlinks=true, pdfstartview=FitV, linkcolor=blue, citecolor=blue]{hyperref}
\usepackage{amssymb,amsmath,amsfonts}
\usepackage{bm}
\usepackage{enumitem}
\usepackage{array}
\usepackage{tikz}
\usetikzlibrary{arrows.meta}
\allowdisplaybreaks

\newcommand{\R}{\mathbb{R}}
\newcommand{\C}{\mathbb{C}}
\newcommand{\Z}{\mathbb{Z}}
\newcommand{\Q}{\mathbb{Q}}
\newcommand{\rar}{\rightarrow}
\newcommand{\mc}{\mathcal}
\newcommand{\mr}{\mathrm}
\newcommand{\Ok}{\mathcal{O}_K}
\newcommand{\Zl}{Z_\Lambda}
\newcommand{\balph}{\bm{\alpha}}
\newcommand{\bp}{\bm{p}}
\newcommand{\bw}{\bm{w}}
\newcommand{\bth}{\bm{\theta}}
\newcommand{\bze}{\bm{0}}
\newcommand{\bc}{\bm{c}}

\DeclareMathOperator{\Arg}{Arg}
\DeclareMathOperator{\Tr}{Tr}
\DeclareMathOperator{\Nm}{N}
\newcommand{\iu}{\mathrm{i}}

\newtheorem*{theorem*}{Theorem}
\newtheorem{theorem}{Theorem}
\newtheorem{lemma}[theorem]{Lemma}
\newtheorem{corollary}[theorem]{Corollary}

\begin{document}

\title[A geometric proof of Peck's theorem]{A geometric proof of Peck's theorem}

\author{Kavita Dhanda, Josh Flynn, Alan Haynes}



\begin{abstract}
Suppose that $d\ge 2$ and that $1,\alpha_1,\ldots ,\alpha_d$ is a basis for a real algebraic number field. A theorem of Peck from 1961 establishes that
\[ \liminf_{n\rar\infty}n\log n\,\|n\alpha_1\|\cdots\|n\alpha_d\|\ < \infty.\]
The goal of this paper is to recast Peck's proof in an intuitive geometric framework, where the result follows from a single application of the Minkowski convex body theorem. We will also explain how a simple modification of this approach leads immediately to new proofs of results of de Mathan, Teuli\'{e}, and Bugeaud regarding the $p$-adic Littlewood conjecture for $d$-tuples of algebraic numbers. Finally, changing the shape of the convex body allows us to make progress on a conjecture raised by Peck in the same paper, about distributing the logarithmic savings unequally among the coordinates. We prove the conjecture for every real biquadratic field with its natural basis, and for an arbitrary number field when the factors involved are of comparable size.
\end{abstract}

\maketitle

\section{Introduction}\label{sec.Intro}
J.~E.~Littlewood conjectured in around 1930 that, for all $\balph=(\alpha_1,\alpha_2)\in\R^2$,
\begin{equation}\label{eqn.LC}
\liminf_{n\rar\infty}n\|n\alpha_1\|\|n\alpha_2\|=0, 
\end{equation}
where $\|x\|$ denotes the distance from $x$ to the nearest integer. In 1955, Cassels and Swinnerton-Dyer \cite{CassSwin1955} proved that \eqref{eqn.LC} holds when $1,\alpha_1,$ and $\alpha_2$ form a basis for a cubic number field. In 1961, Peck \cite{Peck1961}, using a different method, proved the following result.
\begin{theorem}\label{thm.Main}
	Let $d\ge 2$ and suppose that $1,\alpha_1,\ldots,\alpha_d$ is a basis for a real algebraic number
	field. Then there exists a constant $C=C(\alpha_1,\ldots,\alpha_d)$ such that, for all $T\ge 2$, there is a positive integer $n\le T$ with
	\begin{equation}\label{eqn.Peck}
		\|n\alpha_1\|\le Cn^{-1/d},\qquad
		\|n\alpha_i\|\le Cn^{-1/d}(\log T)^{-1/(d-1)}
		\quad (2\le i\le d).
	\end{equation}
\end{theorem}
Note that \eqref{eqn.Peck} implies that
\[\liminf_{n\rar\infty}n\,\log n\,\|n\alpha_1\|\,\|n\alpha_2\|\cdots\|n\alpha_d\|<\infty,\]
which, when $d=2$, is a stronger conclusion than \eqref{eqn.LC}. The goal of this paper is to explain how Peck's proof can be understood in an intuitive way, using basic results from algebraic number theory and the geometry of numbers. For Theorems \ref{thm.Main} and Theorem \ref{thm.Padic} below, we do not claim novelty for the underlying ideas, but instead focus on their organization and presentation. The same framework, with a convex body of a different shape, gives Theorems \ref{thm.Biquad} and \ref{thm.Weighted}, which concern a conjecture raised by Peck in the same paper.

The idea of the proof, which will be explained in more detail shortly, is the following. The lattice $\Lambda^*$ dual to $\Z+\alpha_1\Z+\cdots+\alpha_d\Z$, with respect to the trace form, has a basis element $\alpha_1^*$ of trace zero. Multiplying $\alpha_1^*$ by a unit $u$ of the multiplier ring of $\Lambda$ produces another element of $\Lambda^*$, and every element of $\Lambda^*$ with nonzero trace corresponds to a simultaneous rational approximation to $\balph$. The denominator of the approximation corresponding to $\alpha_1^*u$ is essentially $\alpha_1^*u$, and its vector of approximation errors is a fixed linear function of the conjugates $\sigma_1(\alpha_1^*u),\ldots ,\sigma_d(\alpha_1^*u)$. If the conjugates $\sigma_1(u),\ldots,\sigma_d(u)$ were all equal, the error vector of $\alpha_1^*u$ would be a scalar multiple of the error vector of $\alpha_1^*$, which is the standard basis vector $e_1$. However, they cannot all be equal, but Minkowski's theorem, applied in the lattice of logarithms of units, produces units below any given height for which they are nearly equal, and the logarithm in Theorem \ref{thm.Main} measures how nearly. Figures \ref{fig.PeckCubic} and \ref{fig.PeckTotallyReal} illustrate this in the two cubic cases.

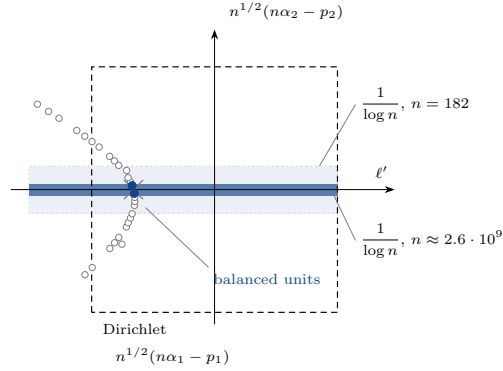
\begin{figure}[b]
\caption{The complex cubic case, $K=\Q(2^{1/3})$, with $\alpha_1=2^{1/3}$, $\alpha_2=2^{2/3}$, $\alpha_1^{*}=2^{2/3}/6$, and fundamental unit $u_1=1+2^{1/3}+2^{2/3}$. (a) The points $\sigma_1(\alpha_1^{*}u_1^{k})$ for $0\le k\le 5$, on a logarithmic radial scale. Every unit satisfies $|\sigma_1(u)|=u^{-1/2}$, so only the argument $\theta(u)$ of $\sigma_1(u)$ can move $\alpha_1^*u$ off the line $\ell$ through $\sigma_1(\alpha_1^{*})$. (b) The normalized errors $n^{1/2}(n\alpha_i-p_i)$ for $n=|\Tr(\alpha_1^{*}u_1^{k})|$, $1\le k\le 32$, where $p_i$ is the nearest integer to $n\alpha_i$. The dashed square is what Dirichlet's theorem gives, and the shaded bands have half-height $1/\log n$ for the two highlighted denominators, $n=223$ and $n\approx 1.4\cdot 10^{18}$. The cross marks $-|\alpha_1^{*}|^{1/2}e_1$.\\}
\label{fig.PeckCubic}
\centering
\resizebox{\textwidth}{!}{%
  \begin{tikzpicture}[>={Stealth[length=4.5pt,width=3.2pt]},line join=round]
    \definecolor{pkAccent}{RGB}{18,68,138}
    \definecolor{pkGrey}{RGB}{125,125,125}
    \tikzset{orb/.style={circle,draw=pkGrey,fill=white,line width=.4pt,inner sep=0pt,minimum size=3.0pt},
             orbhi/.style={circle,draw=pkAccent,fill=pkAccent,inner sep=0pt,minimum size=4.0pt},
             lbl/.style={font=\scriptsize}, lblx/.style={font=\tiny,text=pkGrey},
             lead/.style={line width=.3pt,pkGrey}}

  \begin{scope}
    \fill[pkAccent!9] (0,0) -- (-12.00:3.20) arc[start angle=-12.00,end angle=12.00,radius=3.20] -- cycle;
    \draw[pkAccent!40,line width=.3pt] (0,0) -- (12.00:3.20) (0,0) -- (-12.00:3.20);
    \draw[pkGrey!28,dotted,line width=.4pt] (0,0) circle[radius=2.850];
    \draw[pkGrey!28,dotted,line width=.4pt] (0,0) circle[radius=2.540];
    \draw[pkGrey!28,dotted,line width=.4pt] (0,0) circle[radius=2.230];
    \draw[pkGrey!28,dotted,line width=.4pt] (0,0) circle[radius=1.920];
    \draw[pkGrey!28,dotted,line width=.4pt] (0,0) circle[radius=1.610];
    \draw[pkGrey!28,dotted,line width=.4pt] (0,0) circle[radius=1.300];
    \draw[densely dashed,line width=.55pt] (0,0) -- (3.45,0);
    \node[lbl,anchor=south west] at (3.49,0.06) {$\ell$};
    \draw[pkGrey!75,line width=.45pt] plot[domain=0:1,samples=50,smooth] ({(2.8500+(-0.3100)*\x)*cos((0.0000)+(-146.2010)*\x)},{(2.8500+(-0.3100)*\x)*sin((0.0000)+(-146.2010)*\x)});
    \draw[pkGrey!75,line width=.45pt] plot[domain=0:1,samples=50,smooth] ({(2.5400+(-0.3100)*\x)*cos((-146.2010)+(-146.2010)*\x)},{(2.5400+(-0.3100)*\x)*sin((-146.2010)+(-146.2010)*\x)});
    \draw[pkGrey!75,line width=.45pt] plot[domain=0:1,samples=50,smooth] ({(2.2300+(-0.3100)*\x)*cos((67.5981)+(-146.2010)*\x)},{(2.2300+(-0.3100)*\x)*sin((67.5981)+(-146.2010)*\x)});
    \draw[pkGrey!75,line width=.45pt] plot[domain=0:1,samples=50,smooth] ({(1.9200+(-0.3100)*\x)*cos((-78.6029)+(-146.2010)*\x)},{(1.9200+(-0.3100)*\x)*sin((-78.6029)+(-146.2010)*\x)});
    \draw[pkGrey!75,line width=.45pt] plot[domain=0:1,samples=50,smooth] ({(1.6100+(-0.3100)*\x)*cos((135.1961)+(-146.2010)*\x)},{(1.6100+(-0.3100)*\x)*sin((135.1961)+(-146.2010)*\x)});
    \draw[pkAccent,line width=.55pt] (0,0) -- (-11.005:3.200);
    \draw[pkAccent,line width=.5pt] (3.02,0) arc[start angle=0,end angle=-11.005,radius=3.02];
    \node[lbl,text=pkAccent,anchor=north west] at (3.16,-0.30) {$\theta(u)$};
    \draw[pkAccent,line width=.7pt,->] (2.72,-0.16) .. controls (2.05,-0.95) .. (1.416,-0.418);
    \node[lbl,text=pkAccent,anchor=north] at (2.16,-0.92) {$\times\,\sigma_1(u)$};
    \node[orbhi] (zA0) at (0.000:2.850) {};
    \node[orb] (zA1) at (-146.201:2.540) {};
    \node[orb] (zA2) at (67.598:2.230) {};
    \node[orb] (zA3) at (-78.603:1.920) {};
    \node[orb] (zA4) at (135.196:1.610) {};
    \node[orbhi] (zA5) at (-11.005:1.300) {};
    \node[lbl,anchor=south] at (2.830,0.13) {$\sigma_1(\alpha_1^{*})$};
    \draw[lead] (1.156,-0.318) -- (1.000,-0.580);
    \node[lbl,text=pkAccent,anchor=north east] at (1.06,-0.55) {$\sigma_1(\alpha_1^{*}u_1^{5})$};
    \node[lblx] at (-146.201:2.790) {$u_1^{1}$};
    \node[lblx] at (67.598:2.480) {$u_1^{2}$};
    \node[lblx] at (-78.603:2.170) {$u_1^{3}$};
    \node[lblx] at (135.196:1.860) {$u_1^{4}$};
    \node[circle,fill=black,inner sep=0pt,minimum size=2.4pt] at (0,0) {};
    \node[lbl,anchor=south east] at (-0.06,0.03) {$0$};
    \node[font=\footnotesize,anchor=south] at (0,3.50) {(a)\quad in the field: the $\sigma_1$-plane};
  \end{scope}

  \begin{scope}[xshift=7.95cm]
    \fill[pkAccent!8] (-2.200,-0.407) rectangle (2.200,0.407);
    \draw[pkAccent!45,line width=.3pt,dotted] (-2.200,-0.407) rectangle (2.200,0.407);
    \fill[pkAccent!70] (-2.200,-0.053) rectangle (2.200,0.053);
    \draw[densely dashed,line width=.55pt] (-2.200,-2.200) rectangle (2.200,2.200);
    \draw[->,line width=.4pt] (-2.75,0) -- (3.15,0);
    \draw[->,line width=.4pt] (0,-2.40) -- (0,2.82);
    \node[lbl,anchor=south east] at (3.17,0.07) {$\ell'$};
    \node[lbl,anchor=north] at (0,-2.65) {$n^{1/2}(n\alpha_1-p_1)$};
    \node[lbl,anchor=south west] at (0.12,2.86) {$n^{1/2}(n\alpha_2-p_2)$};
    \node[orb] at (0.572,-0.908) {};
    \node[orb] at (0.175,1.538) {};
    \node[orb] at (-0.862,-1.610) {};
    \node[orb] at (1.264,1.160) {};
    \node[orb] at (0.790,-0.638) {};
    \node[orb] at (-0.078,1.374) {};
    \node[orb] at (-0.661,-1.646) {};
    \node[orb] at (1.176,1.362) {};
    \node[orb] at (-1.294,-0.617) {};
    \node[orb] at (0.974,-0.336) {};
    \node[orb] at (-0.325,1.176) {};
    \node[orb] at (-0.434,-1.618) {};
    \node[orb] at (1.046,1.513) {};
    \node[orb] at (-1.305,-0.897) {};
    \node[orb] at (1.123,-0.023) {};
    \node[orb] at (-0.561,0.934) {};
    \node[orb] at (-0.190,-1.530) {};
    \node[orb] at (0.877,1.609) {};
    \node[orb] at (-1.268,-1.144) {};
    \node[orb] at (1.230,0.292) {};
    \node[orb] at (-0.776,0.659) {};
    \node[orb] at (0.060,-1.387) {};
    \node[orb] at (0.676,1.646) {};
    \node[orb] at (-1.184,-1.349) {};
    \node[orb] at (1.291,0.596) {};
    \node[orb] at (-0.962,0.358) {};
    \node[orb] at (0.308,-1.192) {};
    \node[orb] at (0.451,1.622) {};
    \node[orb] at (-1.057,-1.504) {};
    \node[orb] at (1.306,0.878) {};
    \node[orbhi] (b5) at (-1.235,-0.314) {};
    \node[orbhi] (b32) at (-1.113,0.045) {};
    \node[lbl,text=pkAccent,anchor=north west] at (-1.130,-0.400) {$u_1^{5}$};
    \draw[lead] (-1.163,0.125) -- (-1.200,0.420);
    \node[lbl,text=pkAccent,anchor=south east] at (-1.160,0.440) {$u_1^{32}$};
    \draw[line width=.6pt,pkGrey] (-1.262,-0.13) -- (-1.002,0.13) (-1.262,0.13) -- (-1.002,-0.13);
    \draw[lead] (1.900,0.407) -- (2.500,1.050);
    \draw[lead] (2.140,0.053) -- (2.500,0.980);
    \node[lbl,anchor=west] at (2.460,1.200) {$\pm\dfrac{1}{\log n}$};
    \node[lbl,anchor=north east] at (-2.240,-2.260) {Dirichlet};
    \node[font=\footnotesize,anchor=south] at (0,3.50) {(b)\quad in approximation space};
  \end{scope}
  \end{tikzpicture}
}
\end{figure}

\begin{figure}[t]
\caption{The totally real cubic case, $K=\Q(\rho)$ with $\rho=2\cos(2\pi/7)$, $\alpha_1=\rho$, $\alpha_2=\rho^{2}$, and $\alpha_1^{*}=(-1+2\rho+\rho^{2})/7$. (a) Images of $\alpha_1^{*}$ under totally positive units in the $(\sigma_1,\sigma_2)$-plane, labelled by $\delta(u)=\log|\sigma_1(u)|-\log|\sigma_2(u)|$. Here $\sigma_1(u)\sigma_2(u)=1/u$ is forced, so each image lies on a prescribed hyperbola (drawn on a compressed scale), and only $\delta(u)$ can move it off $\ell$. (b) As in Figure \ref{fig.PeckCubic}, for the totally positive units $u$ with $1<u\le 10^{10}$ and $|\delta(u)|\le 2.6$, here with $n=\Tr(\alpha_1^{*}u)$ and $p_i=\Tr(\alpha_i\alpha_1^{*}u)$ (for $n\ge 7$ these are the nearest integers to $n\alpha_i$); the bands are for $n=182$ and $n\approx 2.6\cdot 10^{9}$.\\}
\label{fig.PeckTotallyReal}
\centering
\resizebox{\textwidth}{!}{%
  \begin{tikzpicture}[>={Stealth[length=4.5pt,width=3.2pt]},line join=round]
    \definecolor{pkAccent}{RGB}{18,68,138}
    \definecolor{pkGrey}{RGB}{125,125,125}
    \tikzset{orb/.style={circle,draw=pkGrey,fill=white,line width=.4pt,inner sep=0pt,minimum size=3.0pt},
             orbhi/.style={circle,draw=pkAccent,fill=pkAccent,inner sep=0pt,minimum size=3.8pt},
             lbl/.style={font=\scriptsize}, lblx/.style={font=\tiny,text=pkGrey},
             lead/.style={line width=.3pt,pkGrey}}

  \begin{scope}
    \fill[pkAccent!9] (1.600,1.350) -- (-2.300,-0.967) -- (-1.289,-1.778) -- cycle;
    \draw[pkAccent!40,line width=.3pt] (1.600,1.350) -- (-2.300,-0.967) (1.600,1.350) -- (-1.289,-1.778);
    \draw[pkGrey!55,dotted,line width=.5pt] plot coordinates{(0.067,-2.100) (0.051,-2.064) (0.035,-2.029) (0.018,-1.994) (0.002,-1.960) (-0.015,-1.926) (-0.032,-1.892) (-0.049,-1.858) (-0.066,-1.825) (-0.083,-1.793) (-0.101,-1.760) (-0.119,-1.728) (-0.137,-1.696) (-0.155,-1.665) (-0.173,-1.634) (-0.191,-1.603) (-0.210,-1.572) (-0.229,-1.542) (-0.248,-1.512) (-0.267,-1.483) (-0.287,-1.454) (-0.306,-1.425) (-0.326,-1.396) (-0.346,-1.368) (-0.367,-1.340) (-0.387,-1.312) (-0.408,-1.284) (-0.429,-1.257) (-0.450,-1.230) (-0.471,-1.204) (-0.493,-1.177) (-0.515,-1.151) (-0.537,-1.126) (-0.559,-1.100) (-0.582,-1.075) (-0.604,-1.050) (-0.627,-1.025) (-0.651,-1.000) (-0.674,-0.976) (-0.698,-0.952) (-0.722,-0.928) (-0.746,-0.905) (-0.770,-0.882) (-0.795,-0.859) (-0.820,-0.836) (-0.845,-0.813) (-0.871,-0.791) (-0.897,-0.769) (-0.923,-0.747) (-0.949,-0.725) (-0.976,-0.704) (-1.002,-0.683) (-1.030,-0.662) (-1.057,-0.641) (-1.085,-0.620) (-1.113,-0.600) (-1.141,-0.580) (-1.170,-0.560) (-1.198,-0.540) (-1.228,-0.521) (-1.257,-0.502) (-1.287,-0.482) (-1.317,-0.464) (-1.347,-0.445) (-1.378,-0.426) (-1.409,-0.408) (-1.441,-0.390) (-1.472,-0.372) (-1.504,-0.354) (-1.537,-0.337) (-1.569,-0.319) (-1.602,-0.302) (-1.636,-0.285) (-1.670,-0.268) (-1.704,-0.251) (-1.738,-0.235) (-1.773,-0.218) (-1.808,-0.202) (-1.844,-0.186) (-1.879,-0.170) (-1.916,-0.155) (-1.952,-0.139) (-1.989,-0.124) (-2.027,-0.109) (-2.065,-0.094) (-2.103,-0.079) (-2.141,-0.064) (-2.181,-0.049) (-2.220,-0.035) (-2.260,-0.021) (-2.300,-0.006)};
    \draw[pkGrey!55,dotted,line width=.5pt] plot coordinates{(0.641,-2.100) (0.626,-2.047) (0.611,-1.994) (0.595,-1.942) (0.580,-1.891) (0.564,-1.841) (0.547,-1.792) (0.531,-1.743) (0.514,-1.695) (0.497,-1.648) (0.479,-1.602) (0.462,-1.556) (0.444,-1.511) (0.426,-1.467) (0.407,-1.424) (0.389,-1.381) (0.370,-1.338) (0.350,-1.297) (0.331,-1.256) (0.311,-1.216) (0.290,-1.176) (0.270,-1.137) (0.249,-1.098) (0.228,-1.060) (0.206,-1.023) (0.184,-0.986) (0.162,-0.950) (0.139,-0.915) (0.117,-0.880) (0.093,-0.845) (0.070,-0.811) (0.045,-0.778) (0.021,-0.745) (-0.004,-0.712) (-0.029,-0.681) (-0.055,-0.649) (-0.081,-0.618) (-0.107,-0.588) (-0.134,-0.558) (-0.161,-0.528) (-0.189,-0.499) (-0.217,-0.471) (-0.245,-0.442) (-0.274,-0.415) (-0.304,-0.387) (-0.334,-0.361) (-0.364,-0.334) (-0.395,-0.308) (-0.426,-0.282) (-0.458,-0.257) (-0.490,-0.232) (-0.523,-0.208) (-0.557,-0.184) (-0.591,-0.160) (-0.625,-0.137) (-0.660,-0.114) (-0.695,-0.091) (-0.731,-0.069) (-0.768,-0.047) (-0.805,-0.025) (-0.843,-0.004) (-0.882,0.017) (-0.921,0.038) (-0.960,0.058) (-1.000,0.078) (-1.041,0.098) (-1.083,0.117) (-1.125,0.136) (-1.168,0.155) (-1.211,0.173) (-1.255,0.192) (-1.300,0.209) (-1.346,0.227) (-1.392,0.244) (-1.439,0.262) (-1.487,0.278) (-1.535,0.295) (-1.585,0.311) (-1.635,0.327) (-1.685,0.343) (-1.737,0.359) (-1.789,0.374) (-1.843,0.389) (-1.897,0.404) (-1.952,0.419) (-2.008,0.433) (-2.064,0.447) (-2.122,0.461) (-2.180,0.475) (-2.240,0.489) (-2.300,0.502)};
    \draw[pkGrey!55,dotted,line width=.5pt] plot coordinates{(0.919,-2.100) (0.905,-2.034) (0.892,-1.969) (0.878,-1.905) (0.864,-1.843) (0.849,-1.781) (0.835,-1.721) (0.820,-1.662) (0.804,-1.604) (0.789,-1.548) (0.773,-1.492) (0.757,-1.437) (0.740,-1.384) (0.723,-1.331) (0.706,-1.280) (0.689,-1.230) (0.671,-1.180) (0.653,-1.131) (0.634,-1.084) (0.615,-1.037) (0.596,-0.991) (0.576,-0.946) (0.556,-0.902) (0.536,-0.859) (0.515,-0.817) (0.494,-0.775) (0.472,-0.734) (0.450,-0.694) (0.428,-0.655) (0.405,-0.616) (0.381,-0.579) (0.357,-0.542) (0.333,-0.505) (0.308,-0.470) (0.283,-0.435) (0.257,-0.401) (0.231,-0.367) (0.204,-0.334) (0.177,-0.302) (0.149,-0.270) (0.120,-0.239) (0.091,-0.208) (0.062,-0.178) (0.032,-0.149) (0.001,-0.120) (-0.030,-0.092) (-0.062,-0.064) (-0.095,-0.037) (-0.128,-0.011) (-0.162,0.016) (-0.196,0.041) (-0.231,0.066) (-0.267,0.091) (-0.304,0.115) (-0.341,0.139) (-0.379,0.162) (-0.418,0.185) (-0.457,0.207) (-0.497,0.229) (-0.538,0.251) (-0.580,0.272) (-0.623,0.293) (-0.666,0.313) (-0.711,0.333) (-0.756,0.352) (-0.802,0.371) (-0.849,0.390) (-0.897,0.409) (-0.946,0.427) (-0.996,0.444) (-1.047,0.462) (-1.098,0.479) (-1.151,0.496) (-1.205,0.512) (-1.260,0.528) (-1.316,0.544) (-1.373,0.559) (-1.431,0.575) (-1.491,0.589) (-1.551,0.604) (-1.613,0.618) (-1.676,0.632) (-1.740,0.646) (-1.805,0.660) (-1.872,0.673) (-1.940,0.686) (-2.009,0.699) (-2.080,0.711) (-2.152,0.723) (-2.225,0.735) (-2.300,0.747)};
    \draw[pkGrey!55,dotted,line width=.5pt] plot coordinates{(1.001,-2.100) (0.988,-2.029) (0.975,-1.959) (0.962,-1.891) (0.948,-1.824) (0.935,-1.759) (0.921,-1.695) (0.907,-1.632) (0.892,-1.571) (0.877,-1.511) (0.862,-1.452) (0.846,-1.394) (0.830,-1.338) (0.814,-1.282) (0.798,-1.228) (0.781,-1.175) (0.764,-1.123) (0.746,-1.072) (0.728,-1.022) (0.710,-0.973) (0.691,-0.926) (0.672,-0.879) (0.652,-0.833) (0.633,-0.788) (0.612,-0.744) (0.591,-0.701) (0.570,-0.658) (0.549,-0.617) (0.526,-0.577) (0.504,-0.537) (0.481,-0.498) (0.457,-0.460) (0.433,-0.423) (0.409,-0.386) (0.384,-0.351) (0.358,-0.315) (0.332,-0.281) (0.305,-0.248) (0.278,-0.215) (0.250,-0.182) (0.222,-0.151) (0.193,-0.120) (0.163,-0.090) (0.133,-0.060) (0.102,-0.031) (0.071,-0.003) (0.039,0.025) (0.006,0.053) (-0.028,0.079) (-0.062,0.105) (-0.097,0.131) (-0.132,0.156) (-0.169,0.181) (-0.206,0.205) (-0.244,0.228) (-0.283,0.251) (-0.322,0.274) (-0.363,0.296) (-0.404,0.318) (-0.446,0.339) (-0.489,0.360) (-0.533,0.380) (-0.578,0.400) (-0.624,0.420) (-0.670,0.439) (-0.718,0.458) (-0.767,0.476) (-0.817,0.494) (-0.868,0.512) (-0.919,0.529) (-0.972,0.546) (-1.026,0.563) (-1.082,0.579) (-1.138,0.595) (-1.196,0.610) (-1.254,0.625) (-1.314,0.640) (-1.376,0.655) (-1.438,0.669) (-1.502,0.683) (-1.567,0.697) (-1.634,0.710) (-1.702,0.724) (-1.771,0.737) (-1.842,0.749) (-1.915,0.762) (-1.989,0.774) (-2.064,0.786) (-2.141,0.797) (-2.220,0.809) (-2.300,0.820)};
    \draw[pkGrey!55,dotted,line width=.5pt] plot coordinates{(1.174,-2.100) (1.163,-2.016) (1.152,-1.934) (1.141,-1.855) (1.130,-1.777) (1.118,-1.701) (1.106,-1.627) (1.094,-1.554) (1.081,-1.484) (1.068,-1.415) (1.055,-1.348) (1.042,-1.282) (1.028,-1.218) (1.013,-1.156) (0.999,-1.095) (0.984,-1.035) (0.968,-0.977) (0.953,-0.921) (0.937,-0.866) (0.920,-0.812) (0.903,-0.759) (0.886,-0.708) (0.868,-0.658) (0.850,-0.609) (0.831,-0.562) (0.812,-0.515) (0.792,-0.470) (0.772,-0.426) (0.752,-0.382) (0.730,-0.340) (0.709,-0.299) (0.687,-0.259) (0.664,-0.220) (0.641,-0.182) (0.617,-0.145) (0.592,-0.108) (0.567,-0.073) (0.541,-0.038) (0.515,-0.005) (0.488,0.028) (0.460,0.060) (0.432,0.092) (0.403,0.122) (0.373,0.152) (0.342,0.181) (0.311,0.210) (0.279,0.237) (0.246,0.264) (0.212,0.291) (0.178,0.317) (0.142,0.342) (0.106,0.366) (0.069,0.390) (0.031,0.413) (-0.008,0.436) (-0.049,0.458) (-0.090,0.480) (-0.132,0.501) (-0.175,0.522) (-0.219,0.542) (-0.264,0.562) (-0.311,0.581) (-0.358,0.599) (-0.407,0.618) (-0.457,0.635) (-0.508,0.653) (-0.561,0.670) (-0.615,0.686) (-0.670,0.702) (-0.726,0.718) (-0.784,0.734) (-0.844,0.749) (-0.905,0.763) (-0.967,0.777) (-1.031,0.791) (-1.096,0.805) (-1.164,0.818) (-1.232,0.831) (-1.303,0.844) (-1.375,0.856) (-1.449,0.868) (-1.525,0.880) (-1.603,0.891) (-1.683,0.902) (-1.765,0.913) (-1.849,0.924) (-1.934,0.934) (-2.023,0.944) (-2.113,0.954) (-2.205,0.964) (-2.300,0.973)};
    \draw[pkGrey!55,dotted,line width=.5pt] plot coordinates{(1.334,-2.100) (1.326,-1.999) (1.317,-1.900) (1.309,-1.805) (1.300,-1.712) (1.291,-1.622) (1.281,-1.535) (1.272,-1.450) (1.262,-1.368) (1.252,-1.288) (1.241,-1.210) (1.230,-1.135) (1.219,-1.062) (1.207,-0.991) (1.196,-0.923) (1.183,-0.856) (1.171,-0.791) (1.158,-0.728) (1.144,-0.667) (1.131,-0.608) (1.116,-0.550) (1.102,-0.494) (1.087,-0.440) (1.071,-0.388) (1.055,-0.337) (1.039,-0.287) (1.022,-0.239) (1.004,-0.192) (0.986,-0.147) (0.967,-0.103) (0.948,-0.060) (0.929,-0.019) (0.908,0.021) (0.887,0.060) (0.866,0.098) (0.844,0.135) (0.821,0.171) (0.797,0.205) (0.773,0.239) (0.748,0.272) (0.722,0.303) (0.695,0.334) (0.668,0.364) (0.640,0.393) (0.611,0.421) (0.581,0.448) (0.550,0.475) (0.518,0.501) (0.485,0.525) (0.452,0.550) (0.417,0.573) (0.381,0.596) (0.344,0.618) (0.306,0.640) (0.267,0.661) (0.227,0.681) (0.185,0.700) (0.142,0.720) (0.098,0.738) (0.053,0.756) (0.006,0.774) (-0.042,0.790) (-0.092,0.807) (-0.143,0.823) (-0.196,0.838) (-0.251,0.853) (-0.307,0.868) (-0.364,0.882) (-0.424,0.896) (-0.485,0.909) (-0.548,0.922) (-0.613,0.935) (-0.680,0.947) (-0.749,0.959) (-0.820,0.970) (-0.893,0.981) (-0.969,0.992) (-1.047,1.003) (-1.127,1.013) (-1.209,1.023) (-1.294,1.032) (-1.382,1.042) (-1.472,1.051) (-1.565,1.060) (-1.661,1.068) (-1.760,1.076) (-1.861,1.084) (-1.966,1.092) (-2.074,1.100) (-2.185,1.107) (-2.300,1.114)};
    \draw[densely dashed,line width=.55pt] (1.600,1.350) -- (-2.300,-1.778);
    \node[lbl,anchor=north west] at (-2.280,-1.758) {$\ell$};
    \draw[->,line width=.4pt] (1.850,1.350) -- (-1.000,1.350);
    \draw[->,line width=.4pt] (1.600,1.600) -- (1.600,-0.700);
    \node[lbl,anchor=south west] at (-1.000,1.390) {$\sigma_1$};
    \node[lbl,anchor=north west] at (1.660,-0.700) {$\sigma_2$};
    \node[circle,fill=black,inner sep=0pt,minimum size=2.4pt] at (1.600,1.350) {};
    \node[lbl,anchor=south west] at (1.650,1.400) {$0$};
    \draw[pkAccent,line width=1.0pt] plot coordinates{(0.246,0.264) (0.229,0.278) (0.213,0.291) (0.195,0.303) (0.178,0.316) (0.161,0.329) (0.143,0.341) (0.125,0.354) (0.107,0.366) (0.088,0.378) (0.070,0.389) (0.051,0.401) (0.032,0.413) (0.012,0.424) (-0.007,0.435) (-0.027,0.447) (-0.047,0.458) (-0.067,0.468) (-0.088,0.479) (-0.109,0.490) (-0.130,0.500) (-0.151,0.511) (-0.173,0.521) (-0.194,0.531) (-0.217,0.541)};
    \draw[lead] (0.062,0.473) -- (0.452,0.963);
    \node[lbl,text=pkAccent,anchor=south] at (0.452,0.983) {$\tfrac12\delta(u)$};
    \draw[pkAccent,line width=.7pt,->] (-0.868,-0.790) .. controls (-0.400,-1.280) .. (0.095,-1.063);
    \node[lbl,text=pkAccent,anchor=north] at (-0.420,-1.340) {$\times\,(\sigma_1(u),\sigma_2(u))$};
    \node[orbhi] at (-0.968,-0.710) {};
    \node[lbl,anchor=north east] at (-1.028,-0.810) {$\bigl(\sigma_1(\alpha_1^{*}),\sigma_2(\alpha_1^{*})\bigr)$};
    \node[orbhi] at (0.195,-1.004) {};
    \node[lblx,text=pkAccent,anchor=west] at (0.305,-1.024) {$\delta=-0.74$};
    \node[orb] at (-1.720,0.642) {};
    \node[lblx,anchor=south] at (-1.720,0.762) {$\delta=+1.32$};
    \node[orb] at (0.831,-1.339) {};
    \node[lblx,anchor=west] at (0.941,-1.339) {$\delta=-1.47$};
    \node[orbhi] at (-0.217,0.541) {};
    \node[lblx,text=pkAccent,anchor=east] at (-0.327,0.561) {$\delta=+0.59$};
    \node[orbhi] at (0.606,0.425) {};
    \node[lblx,text=pkAccent,anchor=west] at (0.716,0.445) {$\delta=-0.15$};
    \node[font=\footnotesize,anchor=south] at (-0.35,3.50) {(a)\quad in the field: the $(\sigma_1,\sigma_2)$-plane};
  \end{scope}

  \begin{scope}[xshift=7.55cm]
    \fill[pkAccent!8] (-3.100,-0.394) rectangle (2.050,0.394);
    \draw[pkAccent!45,line width=.3pt,dotted] (-3.100,-0.394) rectangle (2.050,0.394);
    \fill[pkAccent!70] (-3.100,-0.095) rectangle (2.050,0.095);
    \draw[densely dashed,line width=.55pt] (-2.050,-2.050) rectangle (2.050,2.050);
    \draw[->,line width=.4pt] (-3.40,0) -- (3.00,0);
    \draw[->,line width=.4pt] (0,-2.25) -- (0,2.67);
    \node[lbl,anchor=south east] at (3.02,0.07) {$\ell'$};
    \node[lbl,anchor=north] at (-0.70,-2.50) {$n^{1/2}(n\alpha_1-p_1)$};
    \node[lbl,anchor=south west] at (0.12,2.71) {$n^{1/2}(n\alpha_2-p_2)$};
    \node[orb] at (-1.544,-0.912) {};
    \node[orb] at (-1.467,0.319) {};
    \node[orb] at (-1.471,-0.625) {};
    \node[orb] at (-1.926,0.711) {};
    \node[orb] at (-1.337,-0.265) {};
    \node[orb] at (-2.595,1.191) {};
    \node[orb] at (-1.747,-0.984) {};
    \node[orb] at (-1.600,0.405) {};
    \node[orb] at (-1.440,-0.553) {};
    \node[orb] at (-2.044,0.801) {};
    \node[orb] at (-1.331,-0.197) {};
    \node[orb] at (-2.769,1.305) {};
    \node[orb] at (-2.161,-1.422) {};
    \node[orb] at (-1.404,0.132) {};
    \node[orb] at (-1.668,-0.888) {};
    \node[orb] at (-1.670,0.479) {};
    \node[orb] at (-1.403,-0.476) {};
    \node[orb] at (-2.165,0.891) {};
    \node[orb] at (-1.331,-0.130) {};
    \node[orb] at (-2.956,1.426) {};
    \node[orb] at (-2.040,-1.302) {};
    \node[orb] at (-1.442,0.199) {};
    \node[orb] at (-1.598,-0.798) {};
    \node[orb] at (-1.750,0.555) {};
    \node[orb] at (-1.374,-0.403) {};
    \node[orb] at (-2.298,0.987) {};
    \draw[line width=.6pt,pkGrey] (-1.513,-0.16) -- (-1.193,0.16) (-1.513,0.16) -- (-1.193,-0.16);
    \node[orbhi] at (-1.375,0.066) {};
    \node[orbhi] at (-1.339,-0.064) {};
    \draw[lead] (-1.153,-0.300) -- (-0.100,-1.300);
    \node[lbl,text=pkAccent,anchor=north west] at (-0.140,-1.280) {balanced units};
    \draw[lead] (1.750,0.394) -- (2.350,1.420);
    \node[lbl,anchor=west] at (2.310,1.420) {$\dfrac{1}{\log n}$, $n=182$};
    \draw[lead] (1.990,-0.095) -- (2.350,-0.850);
    \node[lbl,anchor=west] at (2.310,-0.850) {$\dfrac{1}{\log n}$, $n\approx2.6\cdot10^{9}$};
    \node[lbl,anchor=north west] at (-1.990,-2.110) {Dirichlet};
    \node[font=\footnotesize,anchor=south] at (-0.50,3.50) {(b)\quad in approximation space};
  \end{scope}
  \end{tikzpicture}
}
\end{figure}

The same argument, run inside a congruence subgroup of the unit group, gives the following result about denominators in arithmetic progressions.
\begin{theorem}\label{thm.Padic}
	Let $d\ge 2$ and suppose that $1,\alpha_1,\ldots,\alpha_d$ is a basis for a real algebraic number field $K$. Then there is a constant $C=C(\alpha_1,\ldots ,\alpha_d)$ such that, for every positive integer $m$, there are infinitely many positive integers $n$ divisible by $m$ with
	\begin{equation}\label{eqn.Padic1}
		n^{1/d}\max_{1\le i\le d}\|n\alpha_i\|\le C .
	\end{equation}
	Furthermore, for every prime $p$,
	\begin{equation}\label{eqn.Padic2}
		\liminf_{n\rar\infty}\ n^{1/d}(\log n)^{1/(r+z)}\,|n|_p\max_{1\le i\le d}\|n\alpha_i\|<\infty,
	\end{equation}
	where $r+1$ and $2z$ are the numbers of real and non-real embeddings of $K$, so that $r+z$ is the rank of its unit group.
\end{theorem}
Taking $m=p^k$ in \eqref{eqn.Padic1} and letting $k\rar\infty$ shows that \[\liminf_{n}n^{1/d}\allowbreak|n|_p\allowbreak\max_i\|n\alpha_i\|=0,\] for every prime $p$, which is an analogue for $\balph$ of the $p$-adic Littlewood conjecture of de Mathan and Teuli\'e \cite{deMaTeul2004}. A special case of Th\'eor\`eme 3.1 in their paper establishes \eqref{eqn.Padic2} with an exponent of $1/d$ for the logarithm. The statement \eqref{eqn.Padic1}, with $C$ independent of $m$, was proved by Bugeaud \cite[Theorem 1.2]{Buge2015mult}, using the same algebraic construction as here but with a linear recurrence argument in place of Minkowski's theorem, and \cite[Theorem 1.3]{Buge2015mult} asserts \eqref{eqn.Padic2} with $\log n$ in place of $(\log n)^{1/(r+z)}$. However, de Mathan \cite{deMa2024} has pointed out that the argument given there appears to be incomplete for totally real cubic fields, and he proves the exponent $1/2$ in place of $1$ in that case \cite[Theorem 1.5]{deMa2024}. The difficulty is that the unit whose powers are followed in \cite{Buge2015mult} is chosen with a precision which depends on how many powers are needed, and the same issue seems to us to arise whenever $r+z>1$. The exponent $1/(r+z)$ in \eqref{eqn.Padic2} agrees with these results when $d=2$, and improves upon the exponent $1/d$ which follows from \cite[Th\'eor\`eme 3.1]{deMaTeul2004} whenever $K$ has a non-real embedding. We do not know whether the exponent $1$ holds in general, but numerical experiments suggest that it does; see the remark at the end of Section \ref{sec.PadicProof}. Here we also note that, upon making this paper public, Bugeaud and de Mathan informed us that they have recently proved a theorem which provides an improvement over the exponent in the logarithm in \eqref{eqn.Padic2} (see \cite[Theorem 1.2]{BugeDeMath2026}, as well as the comment after their theorem).

We remark that the savings in Theorem \ref{thm.Main} is genuinely logarithmic, in the sense that the factor $(\log T)^{-1/(d-1)}$ cannot be replaced by a positive power of $n^{-1}$: de Mathan \cite{deMa2009} has shown, using Baker's theory of linear forms in logarithms, that for these $\balph$ and for any $c>0$, every $n\ge 2$ with $\max_i\|n\alpha_i\|\le cn^{-1/d}$ satisfies $\max_{i\ge 2}\|n\alpha_i\|\gg_c n^{-1/d}(\log n)^{-\kappa}$ for some $\kappa=\kappa(c)$. An earlier result of the same kind, also obtained from Baker's method, is due to Moshchevitin \cite{Mosh1992}.

The logarithmic savings in \eqref{eqn.Peck} is split equally among $d-1$ of the coordinates, and Peck asked \cite[pp. 197--198]{Peck1961} whether it has to be apportioned in that way:
\begin{quote}
\small
``A number of interesting problems can be raised in connection with (3). In one direction it can be asked whether $n-1$ of the inequalities (2) can be improved with factors which are not all the same; e.g., one might conjecture that we can find infinitely many solutions of the inequalities
\[|q_0\beta_j-q_j\beta_0|<Cq_0^{-1/n}/f_j(q_0)\quad (j=1,\ldots ,n-1),\] \[|q_0\beta_n-q_n\beta_0|<Cq_0^{-1/n},\]
with $f_1(q_0)\cdots f_{n-1}(q_0)=\log q_0$ and $f_j(q_0)\ge 1$ $(j=1,\ldots ,n-1)$.''
\end{quote}

Note that Peck's degree is $n+1$ where ours is $d+1$, and the coordinate carrying no savings is the last one there rather than the first one here. Adjusting for this difference in notation, let us call functions $f_2,\ldots ,f_d$, defined for $T\ge 3$, \textit{admissible} if $f_i(T)\ge 1$ for each $i$ and $f_2(T)\cdots f_d(T)=\log T$ for every $T\ge 3$. Our third result proves the inequalities of Peck's conjecture, for every admissible choice of the $f_i$, for a family of quartic fields.
\begin{theorem}\label{thm.Biquad}
	Let $a$ and $b$ be positive integers, none of $a$, $b$, $ab$ being a perfect square, and set
	\[\alpha_1=\sqrt{ab},\qquad \alpha_2=\sqrt a,\qquad \alpha_3=\sqrt b,\]
	so that $1,\alpha_1,\alpha_2,\alpha_3$ is a basis for the real biquadratic field $\Q(\sqrt a,\sqrt b)$. Then there are constants $C$ and $T_0$, depending only on $a$ and $b$, such that for every admissible pair $f_2,f_3$ and every $T\ge T_0$ there is a positive integer $n\le T$ with
	\begin{equation}\label{eqn.BiquadConc}
		\|n\alpha_1\|\le Cn^{-1/3},\qquad
		\|n\alpha_2\|\le\frac{Cn^{-1/3}}{f_2(T)},\qquad
		\|n\alpha_3\|\le\frac{Cn^{-1/3}}{f_3(T)}.
	\end{equation}
\end{theorem}
We do not assume that $a$ and $b$ are squarefree; the multiplier ring of Section \ref{sec.Background} takes care of this. As an example, taking $a=2$, $b=5$, $f_2(T)=\log T$ and $f_3\equiv 1$ gives the following.
\begin{corollary}\label{cor.TwoFive}
	There are constants $C$ and $T_0$ such that, for every $T\ge T_0$, there is a positive integer $n\le T$ with
	\[\|n\sqrt{10}\|\le Cn^{-1/3},\qquad \|n\sqrt 2\|\le\frac{Cn^{-1/3}}{\log T},\qquad \|n\sqrt 5\|\le Cn^{-1/3}.\]
\end{corollary}
For an arbitrary number field we can prove Peck's conjecture provided that the $f_i$ are not too far apart.
\begin{theorem}\label{thm.Weighted}
	Let $d\ge 2$ and suppose that $1,\alpha_1,\ldots ,\alpha_d$ is a basis for a real algebraic number field. For every $c>0$, there are constants $C$ and $T_0$, depending only on $\alpha_1,\ldots ,\alpha_d$, and $c$, such that, for all admissible $f_2,\ldots ,f_d$ and all $T\ge T_0$ satisfying
	\begin{equation}\label{eqn.WeightCond}
		\max_{2\le i\le d}f_i(T)\ \le\ c\min_{2\le i\le d}f_i(T)^2,
	\end{equation}
	there is a positive integer $n\le T$ with
	\begin{equation}\label{eqn.Weighted}
		\|n\alpha_1\|\le Cn^{-1/d},\qquad
		\|n\alpha_i\|\le\frac{Cn^{-1/d}}{f_i(T)}\quad (2\le i\le d).
	\end{equation}
\end{theorem}
Taking all of the $f_i$ equal recovers Theorem \ref{thm.Main} for $T\ge T_0$, and \eqref{eqn.WeightCond} is satisfied for all large $T$ whenever the $f_i$ are powers of $\log T$ whose exponents differ multiplicatively by a factor less than two. Condition \eqref{eqn.WeightCond} is used in our proof to control the second order term in a linearization argument, and we do not know in general how to remove it.

\noindent\textit{Note in press:} Upon making this manuscript public, we were informed that Bugeaud and de Mathan have very recently proved a version of Theorem \ref{thm.Weighted} (see \cite{BugeDeMath2026}). Interestingly, they also require the same condition \eqref{eqn.WeightCond} as us. This may be taken as independent confirmation that the condition is a natural limitation of the machinery used in our proofs.

We have stated these results with $f_i(T)$ in place of $f_i(q_0)$, and our proofs make no use of any regularity of the $f_i$. Peck does not assume that the $f_i$ are monotone, although that may be the intended reading, and for non-decreasing $f_i$ we have that $f_i(T)\ge f_i(n)$, so \eqref{eqn.BiquadConc} and \eqref{eqn.Weighted} imply the inequalities of Peck's conjecture. However for transparency we mention that, without monotonicity, the passage from $f_i(T)$ to $f_i(n)$ is not available, and the conjecture read literally for arbitrary $f_i$ would be a stronger statement than the one proved here.

We have not found Peck's conjecture treated anywhere in the literature since 1961. Note that the reweighting leaves the product of the bounds in \eqref{eqn.Weighted} unchanged, since $f_2(T)\cdots f_d(T)=\log T$, so results about the product $n\|n\alpha_1\|\cdots\|n\alpha_d\|$ alone cannot distinguish between different choices of the $f_i$, and that the theorem of de Mathan quoted above, being a lower bound for $\max_{i\ge 2}\|n\alpha_i\|$, constrains only the smallest of the $f_i$.

\noindent\textbf{Notation:} We use $O$ and $\ll$ to denote the standard big-oh and
Vinogradov notations, with implied constants depending only on $\alpha_1,\ldots ,\alpha_d$. For $\bm{x}\in\R^d$ we write
$|\bm{x}|=\max_i|x_i|$ for the sup norm and $\|\bm{x}\|=\min_{\bp\in\Z^d}|\bm{x}-\bp|$
for the distance from $\bm{x}$ to the nearest integer point, so that
$\|\bm{x}\|=\max_i\|x_i\|$. For a prime $p$ and $n\in\Z\setminus\{0\}$, $|n|_p$ denotes the $p$-adic absolute value of $n$. The matrix transpose is denoted by $t$, and elements of $\R^{d+1}$ and
$\R^d$ are column vectors. We index the coordinates of $\R^{d+1}$ by $0,\ldots ,d$ and those of $\R^d$ by $1,\ldots ,d$, and in either space $e_i$ denotes the standard basis vector whose $i$th coordinate is $1$; thus $e_0=(1,0,\ldots ,0)^t\in\R^{d+1}$ while $e_1=(1,0,\ldots ,0)^t\in\R^d$, and the vector $(0,e_1)^t\in\R^{d+1}$ of \eqref{eqn.StartPoint} below is the same as $e_1$ when read in $\R^{d+1}$. Throughout, $i$ indexes the coordinates $1,\ldots ,d$ of an approximation, $k$ indexes the embeddings of $K$, and $j$ indexes the pairs of complex conjugate embeddings, always in the combinations $r+j$ and $r+z+j$.

\section{Background from algebraic number theory}\label{sec.Background}

Suppose that $K\subseteq\R$ is an algebraic number field of degree $[K:\Q]=d+1$, with $r+1$ real and $2z$ non-real embeddings into $\C$, so that $d+1=(r+1)+2z$. Label the embeddings as $\sigma_0,\ldots,\sigma_d$ so that
$\sigma_0,\ldots,\sigma_r$ are the real ones, $\sigma_0$ is the identity on $K$, and
$\sigma_{r+j}=\overline{\sigma_{r+z+j}}$ for $1\le j\le z$. Identifying $\C$ with
$\R^2$, the \textit{Minkowski embedding} is the injective map $\sigma:K\rar\R^{d+1}$ defined by
\[
  \sigma(\beta)=\bigl(\sigma_0(\beta),\ldots,\sigma_{r+z}(\beta)\bigr)^{t}.
\]
The \textit{logarithmic embedding} is the homomorphism
$\varphi:K^{\times}\rar\R^{r+z+1}$,
\[
  \varphi(\beta)=\bigl(\log|\sigma_0(\beta)|,\ldots,\log|\sigma_{r+z}(\beta)|\bigr)^{t}.
\]
We write $\Ok$ for the ring of integers of $K$ and $\Ok^{\times}$ for its group of units, and $\Nm$ and $\Tr$ for the norm and the trace from $K$ to $\Q$. We say that a unit $u$ is \textit{totally positive} if $\sigma_k(u)>0$ for $0\le k\le r$.

An additive subgroup of $\R^{d+1}$ is a lattice if it is discrete and cocompact. By extension, we will say that a subset $\Lambda\subseteq K$ is a lattice in $K$ if $\sigma(\Lambda)$ is one. It is a basic fact of algebraic number theory that, if $\alpha_0,\ldots,\alpha_d\in K$ are
$\Q$-linearly independent, then the $\Z$-module
$\alpha_0\Z+\cdots+\alpha_d\Z$ is a lattice in $K$. Another important fact, Dirichlet's unit theorem, guarantees that
$\varphi(\Ok^{\times})$ is a lattice (of rank $r+z$) in the hyperplane of $\R^{r+z+1}$ defined by
\begin{equation}\label{eqn.Hyperplane}
	x_0+\cdots+x_r+2x_{r+1}+\cdots+2x_{r+z}=0 .
\end{equation}

Throughout this paper we fix $\alpha_0=1$ and we suppose that
$\alpha_0,\alpha_1,\ldots,\alpha_d$ is a basis for $K$ over $\Q$. Write
$\balph=(\alpha_1,\ldots,\alpha_d)^{t}$, and
\[
  \Lambda=\Z+\alpha_1\Z+\cdots+\alpha_d\Z.
\]
It is not difficult to show that there are unique $\Q$-linearly independent numbers $\alpha_0^*,\ldots ,\alpha_d^*\in K$ satisfying the equations
\begin{equation}\label{eqn.DualBasis}
\Tr(\alpha_i\alpha_j^*)=\delta_{ij},\quad\text{for} ~ 0\le i,j\le d.
\end{equation}
We refer to the $\Z$-module generated by these numbers as the \textit{dual lattice} to $\Lambda$, which we denote by
\[
  \Lambda^{*}=\alpha_0^{*}\Z+\cdots+\alpha_d^{*}\Z.
\]

Our first key observation, which is implicit in Peck's original proof, is that the dual lattice $\Lambda^*$ naturally encodes information about simultaneous approximations to $\balph$. To make this precise, write $\Sigma^{*}$ for the real
$(d+1)\times(d+1)$ matrix with columns $\sigma(\alpha_0^{*}),\ldots,\sigma(\alpha_d^{*})$, and for $\bm{\beta}\in\R^{d}$ let
\[
u_{\bm{\beta}}=\begin{pmatrix}1&\bze^{t}\\ \bm{\beta}&I_d\end{pmatrix}
\ \in\ \mr{SL}_{d+1}(\R).
\]
A simple computation (cf. \cite[Lemma 19]{DhanHayn2025}), using only the defining property \eqref{eqn.DualBasis}, gives the following result.
\begin{lemma}\label{lem.Frame}
	With $\balph$ as above, the matrix 
	\[\mc{R}=u_{-\balph}(\Sigma^*)^{-1}\]
	has first column $e_0=(1,0,\ldots,0)^{t}$.
\end{lemma}
We will call the matrix $\mc{R}$ of Lemma \ref{lem.Frame} the \textit{rectifying frame} of $\balph$. Note that for $s=q\alpha_0^{*}+p_1\alpha_1^{*}+\cdots+p_d\alpha_d^{*}\in\Lambda^*$, we have that
\begin{equation}\label{eqn.FrameAction}
	\mc{R}\sigma(s)=(q,\ \bp-q\balph)^{t},
\end{equation}
since $(\Sigma^*)^{-1}\sigma(s)=(q,\bp)^t$ is the coordinate vector of $s$ with respect to the dual basis. In particular, taking $s=\alpha_1^*$,
\begin{equation}\label{eqn.StartPoint}
	\mc{R}\sigma(\alpha_1^*)=(0,e_1)^{t}.
\end{equation}
Lemma \ref{lem.Frame} tells us that $\mc{R}$ can be written in block form as
\begin{equation}\label{eqn.BlockForm}
	\mc{R}=\begin{pmatrix}1&\bc^{t}\\ \bze&\tilde{\mc{R}}\end{pmatrix},
\end{equation}
with $\bc\in\R^d$ and $\tilde{\mc{R}}$ a $d\times d$ matrix. In other words, the vector $\bp-q\balph$ of approximation errors in \eqref{eqn.FrameAction} is a fixed linear function of the last $d$ coordinates of $\sigma(s)$ (i.e. of the conjugates $\sigma_1(s),\ldots ,\sigma_d(s)$), and does not depend on $s=\sigma_0(s)$ itself.

We will produce elements of $\Lambda^*$ by multiplying $\alpha_1^*$ by units. For $u\in K$ let $M_u$ denote the real $(d+1)\times(d+1)$ matrix which represents multiplication by $u$ in Minkowski coordinates, so that $\sigma(\beta u)=M_u\sigma(\beta)$ for all $\beta\in K$. The matrix $M_u$ is block diagonal, with $1\times 1$ blocks $\sigma_k(u)$ for $0\le k\le r$, followed by the $2\times 2$ blocks
\[
	\begin{pmatrix}\mr{Re}\,\sigma_{r+j}(u)&-\mr{Im}\,\sigma_{r+j}(u)\\ \mr{Im}\,\sigma_{r+j}(u)&\mr{Re}\,\sigma_{r+j}(u)\end{pmatrix},\quad 1\le j\le z,
\]
representing multiplication by $\sigma_{r+j}(u)$ on $\C\cong\R^2$.

In order to identify a full rank subgroup of $\Ok^\times$ which acts on $\Lambda$ and $\Lambda^*$, we first define the \textit{multiplier ring} $\Zl$ of $\Lambda$ by
\[
  \Zl=\{\gamma\in\Ok:\ \gamma\Lambda\subseteq\Lambda\}.
\]
It is not difficult to verify the following basic lemma (see \cite{Peck1961} and \cite[Lemma 13]{DhanHayn2025}).
\begin{lemma}\label{lem.Multiplier}
Let $\Lambda\subseteq K$ be a lattice. Then $\gamma\Lambda^{*}\subseteq\Lambda^{*}$ for
every $\gamma\in\Zl$, and $u\Lambda^{*}=\Lambda^{*}$ for every $u\in\Zl^{\times}$.
Furthermore $\Zl^{\times}$ is a subgroup of $\Ok^{\times}$ of maximal rank $r+z$, so
that $\varphi(\Zl^{\times})$ is a lattice in the hyperplane \eqref{eqn.Hyperplane}.
\end{lemma}

The following lemma, which is where Minkowski's theorem is used, is the main gear that makes the proofs of Theorems \ref{thm.Main} and \ref{thm.Padic} work. It shows that units $u>1$ whose conjugates are all close to $u^{-1/d}$ can be found below every sufficiently large bound.

\begin{lemma}\label{lem.Balance}
Let $d\ge 2$. There is a constant $T_1=T_1(\Lambda)$ such that, for all $T\ge T_1$, there exists $u\in\Zl^\times$ with $1<u\le T$ and
\begin{equation}\label{eqn.Balanced}
	\bigl|u^{1/d}\sigma_k(u)-1\bigr|\ \ll\ (\log T)^{-1/(d-1)}\quad (1\le k\le d).
\end{equation}
\end{lemma}

\begin{proof}
Let $P=\{u\in\Zl^\times : u>0\}$. Since $\pm 1$ are the only roots of unity in $K$, we have that $\varphi(P)=\varphi(\Zl^\times)$, and that $\varphi$ is injective on $P$. For $u\in P$ and $1\le k\le r+z$ write
\[
	\lambda_k(u)=\log|\sigma_k(u)|+\frac{1}{d}\log u ,
\]
so that $\lambda_k(u)=0$ if and only if $|\sigma_k(u)|=u^{-1/d}$. Since $r+2z=d$, equation \eqref{eqn.Hyperplane} gives
\[
	\lambda_1(u)+\cdots+\lambda_r(u)+2\lambda_{r+1}(u)+\cdots +2\lambda_{r+z}(u)=0,
\]
so that $\lambda_{r+z}(u)$ is determined by $\lambda_1(u),\allowbreak\ldots ,\allowbreak\lambda_{r+z-1}(u)$, and also
\[|\lambda_{r+z}(u)|\le d\max_{k<r+z}|\lambda_k(u)|.\] As a special case, note that $\lambda_{r+z}(u)=0$ if $r+z=1$.

Now let $\Gamma\subseteq\R^{d}=\R\times\R^{r+z-1}\times\R^{z}$ be the collection of all vectors
\[
	\bigl(\log u,\ \lambda_1(u),\ldots ,\lambda_{r+z-1}(u),\ \bth\bigr)
\]
with $u\in P$ and $\bth\in\R^z$ satisfying $\theta_j\equiv\Arg\sigma_{r+j}(u)~\mathrm{mod}~2\pi$ for $1\le j\le z$. Since each of $\log u$, $\lambda_k(u)$, and $\Arg\sigma_{r+j}(u)~\mathrm{mod}~2\pi$ is a homomorphism in $u$, $\Gamma$ is a subgroup of $\R^d$. It is discrete, because if all coordinates of a point of $\Gamma$ are less than $\delta$ in absolute value then, since \[\log|\sigma_k(u)|=\lambda_k(u)-\frac1d\log u\quad\text{ and }\quad |\lambda_{r+z}(u)|\le d\max_{k<r+z}|\lambda_k(u)|,\] all coordinates of $\varphi(u)$ are less than $(d+1)\delta$ in absolute value. For $\delta$ small enough the discreteness of $\varphi(P)$ forces $u=1$, and then $\bth\in 2\pi\Z^z$ forces $\bth=\bze$. Furthermore $\Gamma$ contains $\{\bze\}\times 2\pi\Z^z$, and the quotient of $\Gamma$ by this subgroup is isomorphic to $P\cong\Z^{r+z}$, so $\Gamma$ has rank $(r+z)+z=d$. Therefore $\Gamma$ is a lattice in $\R^d$. Write $V$ for its covolume.

Let
\[
	\varepsilon=\frac12\left(\frac{2^dV}{\log T}\right)^{1/(d-1)},
\]
and let $T_1$ be large enough that $\varepsilon\le 1/(4d)$ for $T\ge T_1$. The box
\[
	B=\left\{\bm{x}\in\R^d : |x_0|\le\tfrac12\log T,\ |x_k|\le\varepsilon\ \text{for}\ 1\le k\le d-1\right\}
\]
is compact, convex, and symmetric about the origin, and its volume is \[\log T\cdot(2\varepsilon)^{d-1}=2^dV.\] By Minkowski's convex body theorem, $B$ contains a nonzero point of $\Gamma$, which we write as $(\log v,\lambda_1(v),\ldots ,\lambda_{r+z-1}(v),\bth)$ with $v\in P$. Since $\Gamma$ is symmetric we may replace this point by its negative, which replaces $v$ by $v^{-1}$ and $\bth$ by $-\bth$, so we may assume that $\log v\ge 0$. In fact $\log v>0$. Indeed, if $\log v=0$ then $v=1$ and every $\lambda_k(v)=0$, but this does not by itself make the point the origin, because its last $z$ coordinates are only defined modulo $2\pi$; what rules this out is that $\bth\in 2\pi\Z^z$ together with $|\bth|\le\varepsilon<2\pi$ forces $\bth=\bze$. Therefore $1<v\le T^{1/2}$.

Finally, let $u=v^2$, so that $u\in\Zl^\times$ and $1<u\le T$. For $1\le k\le r$ we have that $\sigma_k(u)=\sigma_k(v)^2>0$, and therefore $u^{1/d}\sigma_k(u)=e^{\lambda_k(u)}$ with $|\lambda_k(u)|=2|\lambda_k(v)|\le 2d\varepsilon$. For $1\le j\le z$ we have that $\Arg\sigma_{r+j}(u)\equiv 2\theta_j~\mathrm{mod}~2\pi$ with $|2\theta_j|\le 2\varepsilon<\pi$, and therefore $u^{1/d}\sigma_{r+j}(u)=e^{\lambda_{r+j}(u)+2\iu\theta_j}$, with $|\lambda_{r+j}(u)+2\iu\theta_j|\le 2d\varepsilon+2\varepsilon\le 4d\varepsilon$; the same bound holds for the conjugates $\sigma_{r+z+j}(u)$. In all cases $u^{1/d}\sigma_k(u)=e^{\zeta_k}$ with $|\zeta_k|\le 4d\varepsilon\le 1$, so that $|u^{1/d}\sigma_k(u)-1|\le |\zeta_k|e^{|\zeta_k|}\le 12d\varepsilon\ll(\log T)^{-1/(d-1)}$.
\end{proof}

\section{Proof of Theorem \ref{thm.Main}}\label{sec.Proof}

Suppose that $u\in\Zl^\times$ satisfies $u>1$ and
\begin{equation}\label{eqn.UnitHyp}
	\sigma_k(u)=u^{-1/d}(1+\epsilon_k),\quad |\epsilon_k|\le\varepsilon,\quad\text{for}~1\le k\le d,
\end{equation}
for some $0<\varepsilon\le 1$. Then the matrix $M_u$ of Section \ref{sec.Background} can be written as
\[
	M_u=\begin{pmatrix}u&\bze^t\\ \bze&u^{-1/d}I_d\end{pmatrix}+u^{-1/d}E,
\]
where $E$ is block diagonal with blocks $0$, then $\epsilon_1,\ldots ,\epsilon_r$, then the $2\times 2$ matrices representing multiplication by $\epsilon_{r+1},\ldots ,\epsilon_{r+z}$ on $\C\cong\R^2$. Each row of $E$ has at most two nonzero entries, each of absolute value at most $\varepsilon$, so $|E\bm{x}|\le 2\varepsilon|\bm{x}|$ for all $\bm{x}\in\R^{d+1}$.

Now let $s=\alpha_1^*u$, which by Lemma \ref{lem.Multiplier} is an element of $\Lambda^*$, and write $\mc{R}\sigma(s)=(q,\bp-q\balph)^t$ as in \eqref{eqn.FrameAction}, with $q\in\Z$ and $\bp\in\Z^d$. Writing $\sigma(\alpha_1^*)=(\alpha_1^*,\bw)^t$ with $\bw\in\R^d$, equations \eqref{eqn.StartPoint} and \eqref{eqn.BlockForm} give $\alpha_1^*+\bc^t\bw=0$ and $\tilde{\mc{R}}\bw=e_1$. Therefore
\begin{align*}
	\begin{pmatrix}q\\ \bp-q\balph\end{pmatrix}=\mc{R}M_u\sigma(\alpha_1^*)
	&=\mc{R}\begin{pmatrix}u\alpha_1^*\\ u^{-1/d}\bw\end{pmatrix}+u^{-1/d}\mc{R}E\sigma(\alpha_1^*)\\
	&=\begin{pmatrix}\alpha_1^*(u-u^{-1/d})\\ u^{-1/d}e_1\end{pmatrix}+u^{-1/d}\mc{R}E\sigma(\alpha_1^*).
\end{align*}
Since $|\mc{R}E\sigma(\alpha_1^*)|\ll\varepsilon\le 1$, this gives
\begin{equation}\label{eqn.KeyEstimate}
\begin{gathered}
	q=\alpha_1^*u+O(u^{-1/d}),\qquad
	|p_1-q\alpha_1-u^{-1/d}|\ll\varepsilon u^{-1/d},\\
	|p_i-q\alpha_i|\ll\varepsilon u^{-1/d}\quad(2\le i\le d).
\end{gathered}
\end{equation}

Now let $T\ge 2$. By the first estimate in \eqref{eqn.KeyEstimate}, there is a constant $C'=C'(\balph)\ge 2$ such that $|q|\le C'u$ whenever $u>1$ satisfies \eqref{eqn.UnitHyp}. If $T\ge C'^{2}$ and $T/C'\ge T_1$, then we may apply Lemma \ref{lem.Balance} with $T/C'$ in place of $T$, to find $u\in\Zl^\times$ with $1<u\le T/C'$ satisfying \eqref{eqn.UnitHyp} with
\[
	\varepsilon\ll (\log (T/C'))^{-1/(d-1)}\le (\tfrac12\log T)^{-1/(d-1)}\ll(\log T)^{-1/(d-1)} .
\]
We claim that, if $T$ is also large enough that $\varepsilon\le 1$ and that the implied constant in the last estimate of \eqref{eqn.KeyEstimate} times $\varepsilon$ is less than $1$, then $q\not= 0$. Suppose by way of contradiction that $q=0$. Then, for $2\le i\le d$, the integers $p_i=p_i-q\alpha_i$ have absolute value less than $1$, so they are all zero, and $s=p_1\alpha_1^*$. But $s=\alpha_1^*u$, so $u=p_1$ is a rational integer which is also a unit, contradicting $u>1$.

Let $n=|q|$, so that $1\le n\le C'u\le T$, and let $p_i'=\mr{sgn}(q)p_i$ for $1\le i\le d$, so that $\|n\alpha_i\|\le|n\alpha_i-p_i'|=|p_i-q\alpha_i|$. Since $u^{-1/d}\le C'^{1/d}n^{-1/d}$, the second and third estimates in \eqref{eqn.KeyEstimate} give
\[
\begin{gathered}
	\|n\alpha_1\|\le u^{-1/d}(1+O(\varepsilon))\ll n^{-1/d},\\
	\|n\alpha_i\|\ll\varepsilon u^{-1/d}\ll n^{-1/d}(\log T)^{-1/(d-1)}\quad (2\le i\le d),
\end{gathered}
\]
which is \eqref{eqn.Peck} for all $T\ge T_2$, where $T_2=T_2(\balph)$ is the largest of the thresholds imposed on $T$ above. For $2\le T<T_2$ the bounds \eqref{eqn.Peck} hold with $n=1$, after enlarging $C$ if necessary, since $\|\alpha_i\|\le 1/2$. This completes the proof of Theorem \ref{thm.Main}.

\section{Proof of Theorem \ref{thm.Padic}}\label{sec.PadicProof}

Let $m$ be a positive integer and let
\[
	U_m=\{u\in\Zl^\times : u\equiv 1~\mathrm{mod}~m\Zl\}.
\]
Since $U_m$ is the kernel of the canonical homomorphism from $\Zl^\times$ to $(\Zl/m\Zl)^\times$, it is a subgroup of $\Zl^\times$ of finite index $I_m\le|(\Zl/m\Zl)^\times|\le m^{d+1}$. If $u=1+m\gamma\in U_m$ then, since $\Tr(\alpha_1^*)=0$ and $\alpha_1^*\gamma\in\Lambda^*$ by Lemma \ref{lem.Multiplier}, we have that
\begin{equation}\label{eqn.Divisible}
	q=\Tr(\alpha_1^*u)=m\Tr(\alpha_1^*\gamma)\in m\Z .
\end{equation}

Now we run the proof of Lemma \ref{lem.Balance} with $P\cap U_m$ in place of $P$. The lattice $\Gamma_m$ which this produces is a sublattice of $\Gamma$ of index $[P:P\cap U_m]\le I_m$, so its covolume is at most $I_mV$, and the proof goes through verbatim to show that there are constants $c_1,c_2>0$, depending only on $\Lambda$, such that for all $T$ with $\log T\ge c_1I_m$ there is $u\in U_m$ with $1<u\le T$ and
\[
	\bigl|u^{1/d}\sigma_k(u)-1\bigr|\le c_2\left(\frac{I_m}{\log T}\right)^{1/(d-1)}\quad(1\le k\le d).
\]
Feeding this unit into the proof of Theorem \ref{thm.Main} in place of the unit from Lemma \ref{lem.Balance}, and enlarging $c_1$ if necessary so that the thresholds imposed on $T$ there are met, we find that for all $T$ with $\log T\ge c_1I_m$ there is an integer $1\le n\le T$ which, by \eqref{eqn.Divisible}, is divisible by $m$, and which satisfies
\begin{equation}\label{eqn.PadicKey}
	\|n\alpha_1\|\ll n^{-1/d},\qquad \|n\alpha_i\|\ll n^{-1/d}\left(\frac{I_m}{\log T}\right)^{1/(d-1)}\ll n^{-1/d}\quad(2\le i\le d),
\end{equation}
with implied constants which do not depend on $m$. This gives \eqref{eqn.Padic1} for one value of $n$, and letting $T\rar\infty$ gives infinitely many, since for fixed $n$ the left hand side of the second inequality in \eqref{eqn.PadicKey} is positive while the right hand side tends to zero.

For \eqref{eqn.Padic2}, let $m=p^k$. Every element of the kernel of the canonical homomorphism from $(\Zl/p^k\Zl)^\times$ to $(\Zl/p\Zl)^\times$ has order dividing $p^{k-1}$, since $x\equiv 1~\mathrm{mod}~p^j\Zl$ with $j\ge 1$ implies that $x^p\equiv 1~\mathrm{mod}~p^{j+1}\Zl$, and $(\Zl/p\Zl)^\times$ has fewer than $p^{d+1}$ elements. Therefore every element of $(\Zl/p^k\Zl)^\times$ has order less than $p^{k+d}$. By Dirichlet's unit theorem and Lemma \ref{lem.Multiplier}, $\Zl^\times$ is generated by $-1$ together with $r+z$ further elements, so its image in $(\Zl/p^k\Zl)^\times$ has order at most $2p^{(k+d)(r+z)}$, which is to say
\[
	I_{p^k}\le 2p^{(k+d)(r+z)}\ll_p p^{k(r+z)} .
\]
Taking $T=T_k=\exp(c_1I_{p^k})$ in \eqref{eqn.PadicKey}, we obtain integers $n_k\le T_k$, divisible by $p^k$, with $n_k^{1/d}\max_i\|n_k\alpha_i\|\ll 1$ and
\[
	|n_k|_p\le p^{-k}\ll_p I_{p^k}^{-1/(r+z)}=\left(\frac{\log T_k}{c_1}\right)^{-1/(r+z)}\ll(\log n_k)^{-1/(r+z)} .
\]
Since $n_k\ge p^k$, we have that $n_k\rar\infty$, and \eqref{eqn.Padic2} follows. This completes the proof of Theorem \ref{thm.Padic}.

We remark that the exponent $1/(r+z)$ comes entirely from the index of $U_{p^k}$ in $\Zl^\times$. What the proof of Theorem \ref{thm.Main} actually requires is not that $u\in U_{p^k}$ but only that $p^k$ divides $\Tr(\alpha_1^*u)$, and this is satisfied by many more units: for any fixed tolerance $\varepsilon$ in \eqref{eqn.UnitHyp}, the number of units satisfying \eqref{eqn.UnitHyp} with $0<\log u\le X$ is $\asymp X$, by counting points of $\Gamma$ in a box, and if the integers $\Tr(\alpha_1^*u)$ were equidistributed modulo $p^k$ over these units then one would expect a suitable $u$ with $\log u\ll_p p^k$, and hence the exponent $1$ (cf. discussion after statement of Theorem \ref{thm.Padic} above). Numerical experiments in the fields generated by roots of $x^3-x-1$, $x^3-3x-1$, $x^3-4x-1$ and $x^4-x-1$, with $p=2$ and $p=3$, are consistent with this expectation.

\section{Unequal weights: proofs of Theorems \ref{thm.Weighted} and \ref{thm.Biquad}}\label{sec.Weighted}

The proof of Theorem \ref{thm.Main} in Section \ref{sec.Proof} uses the unit supplied by Lemma \ref{lem.Balance} only through the single bound $|\epsilon_k|\le\varepsilon$, and consequently it bounds all of the errors $p_i-q\alpha_i$, $2\le i\le d$, by the same quantity. To distribute the savings unevenly we need an exact description of these errors as a function of the deviation $\bm\epsilon$, and a version of Lemma \ref{lem.Balance} in which the convex body is allowed to have different side lengths in different directions. We will prove Theorem \ref{thm.Weighted} first.

\noindent\emph{The error map.}\quad Let
\[
	V_\R=\bigl\{\bm\xi\in\C^d:\ \xi_k\in\R\ (1\le k\le r),\ \xi_{r+z+j}=\overline{\xi_{r+j}}\ (1\le j\le z)\bigr\},
\]
a real vector space of dimension $r+2z=d$, on which we write $|\bm\xi|=\max_k|\xi_k|$, and let
\[
	H=\Bigl\{\bm\xi\in V_\R:\ \xi_1+\cdots+\xi_d=0\Bigr\},
\]
a subspace of dimension $d-1$. For $\bm\xi\in V_\R$ let $\iota(\bm\xi)\in\R^d$ be the vector obtained by listing $\xi_1,\ldots ,\xi_r$ and then, for each $1\le j\le z$, the real and imaginary parts of $\xi_{r+j}$. This is the order in which coordinates are taken in $\sigma$, so that for $\beta\in K$ the vector $\iota(\sigma_1(\beta),\ldots ,\sigma_d(\beta))$ consists of the last $d$ coordinates of $\sigma(\beta)$. Define $\Psi:V_\R\rar\R^d$ by
\begin{equation}\label{eqn.PsiDef}
	\Psi(\bm\xi)=\tilde{\mc R}\,\iota\bigl(\sigma_1(\alpha_1^*)\xi_1,\ldots ,\sigma_d(\alpha_1^*)\xi_d\bigr).
\end{equation}

\begin{lemma}\label{lem.ErrorMap}
The map $\Psi$ is a linear isomorphism from $V_\R$ onto $\R^d$, and its components are given by
\begin{equation}\label{eqn.PsiTrace}
	\Psi_i(\bm\xi)=\sum_{k=1}^{d}\sigma_k(\alpha_1^*)\bigl(\sigma_k(\alpha_i)-\alpha_i\bigr)\xi_k\qquad(1\le i\le d).
\end{equation}
In particular $\Psi(\bm 1)=e_1$, where $\bm 1=(1,\ldots ,1)$, and $\Psi'=(\Psi_2,\ldots ,\Psi_d)$ restricts to a linear isomorphism from $H$ onto $\R^{d-1}$.
\end{lemma}

\begin{proof}
Equation \eqref{eqn.FrameAction} holds, with the same proof, for every $\beta\in K$, with $q=\Tr(\beta)$ and $p_i=\Tr(\alpha_i\beta)$ its coordinates with respect to the dual basis. Together with \eqref{eqn.BlockForm} this gives
\[
	\tilde{\mc R}\,\iota\bigl(\sigma_1(\beta),\ldots ,\sigma_d(\beta)\bigr)
	=\bigl(\Tr(\alpha_i\beta)-\Tr(\beta)\alpha_i\bigr)_{i=1}^{d}
	=\Bigl(\sum_{k=1}^{d}\bigl(\sigma_k(\alpha_i)-\alpha_i\bigr)\sigma_k(\beta)\Bigr)_{i=1}^{d},
\]
the terms with $k=0$ having canceled. Both sides of this equation are $\R$-linear in the vector $(\sigma_1(\beta),\ldots ,\sigma_d(\beta))$, and these vectors span $V_\R$, so the identity holds with an arbitrary $\bm\xi\in V_\R$ in place of $(\sigma_k(\beta))_k$. Replacing $\xi_k$ by $\sigma_k(\alpha_1^*)\xi_k$ gives \eqref{eqn.PsiTrace}.

Taking $\bm\xi=\bm 1$ in \eqref{eqn.PsiTrace}, and using $\Tr(\alpha_1^*)=\Tr(\alpha_0\alpha_1^*)=0$, we have that $\Psi_i(\bm 1)=\Tr(\alpha_i\alpha_1^*)=\delta_{i1}$, which is to say that $\Psi(\bm 1)=e_1$. Next, the map $\bm\xi\mapsto(\sigma_k(\alpha_1^*)\xi_k)_k$ is a bijection of $V_\R$, since every $\sigma_k(\alpha_1^*)$ is nonzero and $\sigma_{r+z+j}(\alpha_1^*)=\overline{\sigma_{r+j}(\alpha_1^*)}$, the map $\iota$ is a bijection onto $\R^d$, and $\tilde{\mc R}$ is invertible because $\det\mc R=\det\tilde{\mc R}$ by \eqref{eqn.BlockForm}. Therefore $\Psi$ is an isomorphism, and
\[
	\ker\Psi'=\Psi^{-1}(\R e_1)=\R\bm 1 .
\]
Since $\bm 1\notin H$, $\Psi'$ is injective on $H$, and $\dim H=d-1$ finishes the proof.
\end{proof}

Now suppose that $u\in\Zl^\times$ satisfies \eqref{eqn.UnitHyp}. The vector $E\sigma(\alpha_1^*)$ of Section \ref{sec.Proof} has first coordinate $0$, and its remaining coordinates are those of $\iota(\sigma_1(\alpha_1^*)\epsilon_1,\ldots ,\sigma_d(\alpha_1^*)\epsilon_d)$, so by \eqref{eqn.BlockForm} the last $d$ coordinates of $\mc RE\sigma(\alpha_1^*)$ are $\Psi(\bm\epsilon)$. The computation of Section \ref{sec.Proof} therefore gives, in place of the last two estimates of \eqref{eqn.KeyEstimate}, the exact identity
\begin{equation}\label{eqn.ExactError}
	p_i-q\alpha_i=u^{-1/d}\bigl(\delta_{i1}+\Psi_i(\bm\epsilon)\bigr)\qquad(1\le i\le d),
\end{equation}
while the first estimate of \eqref{eqn.KeyEstimate} for $q$ is unchanged. Substituting $\epsilon_k=u^{1/d}\sigma_k(u)-1$ into \eqref{eqn.PsiTrace} and using $\Psi(\bm 1)=e_1$ shows that \[\Psi_i(\bm\epsilon)=u^{1/d}\bigl(\Tr(\alpha_i\alpha_1^*u)-\alpha_i\Tr(\alpha_1^*u)\bigr)-\delta_{i1},\]
which is \eqref{eqn.ExactError} again, written in a different form. Up to this normalization, the components of $\Psi(\bm\epsilon)$ are the linear forms in the conjugates of $u$ that Peck works with directly in \cite{Peck1961}. The rectifying frame adds a geometric point of view: choosing a unit amounts to choosing a point $\bm\epsilon$, and the $d-1$ errors we wish to make small are the coordinates of that point in the system $\Psi_2,\ldots ,\Psi_d$.

\noindent\emph{Balanced units in a box of prescribed shape.}\quad We will use the notation from the proof of Lemma \ref{lem.Balance}. For a totally positive $u\in\Zl^\times$ with $u>1$, write $\bm\zeta(u)\in V_\R$ for the vector with
\[
	\zeta_k(u)=\lambda_k(u)\ \ (1\le k\le r),\qquad
	\zeta_{r+j}(u)=\lambda_{r+j}(u)+\iu\Arg\sigma_{r+j}(u)\ \ (1\le j\le z).
\]
Then we have that
\begin{equation}\label{eqn.ZetaDef}
	u^{1/d}\sigma_k(u)=e^{\zeta_k(u)}\quad(1\le k\le d),\qquad\text{and hence}\qquad \epsilon_k=e^{\zeta_k(u)}-1 .
\end{equation}
Since $|\Nm(u)|=1$ and $\sigma_0(u)=u$, the real parts of the $\zeta_k(u)$ sum to $\log(1/u)+\log u=0$, while the imaginary parts cancel in conjugate pairs; thus $\bm\zeta(u)\in H$.

\begin{lemma}\label{lem.BalanceBox}
Let $d\ge 2$ and let $Y:H\rar\R^{d-1}$ be a linear isomorphism. There are positive constants $L=L(Y)$ and $J=J(Y,\Lambda)$ with the following property. Suppose that $T>1$ and that $\eta_2,\ldots ,\eta_d>0$ satisfy
\begin{equation}\label{eqn.BoxVol}
	\prod_{i=2}^{d}\eta_i=\frac{J}{\log T}\qquad\text{and}\qquad L\max_{2\le i\le d}\eta_i\le\frac14 .
\end{equation}
Then there is a totally positive $u\in\Zl^\times$ with $1<u\le T$ and
\begin{equation}\label{eqn.BoxConc}
	\bigl|Y_i(\bm\zeta(u))\bigr|\le 2\eta_i\quad(2\le i\le d),\qquad
	\bigl|\bm\zeta(u)\bigr|\le 2L\max_{2\le i\le d}\eta_i\le\frac12 .
\end{equation}
\end{lemma}

\begin{proof}
We use the notation of Lemma \ref{lem.Balance}, and follow its proof. Let $\Theta:\R^{r+z-1}\times\R^z\rar H$ be the map which sends $(\lambda_1,\ldots ,\lambda_{r+z-1},\bth)$ to the vector $\bm\xi$ with $\xi_k=\lambda_k$ for $k\le r$ and $\xi_{r+j}=\lambda_{r+j}+\iu\theta_j$, where $\lambda_{r+z}$ is defined by
\[
	\lambda_1+\cdots+\lambda_r+2\lambda_{r+1}+\cdots+2\lambda_{r+z}=0 .
\]
This relation is exactly the condition $\sum_k\xi_k=0$, so $\Theta$ is a linear isomorphism onto $H$, and $Y\circ\Theta$ is a linear isomorphism of $\R^{d-1}$. Therefore \[\max_{2\le i\le d}\bigl|Y_i(\Theta(\lambda,\bth))\bigr|\] is a norm on $\R^{d-1}$ and, all norms on $\R^{d-1}$ being equivalent, we may take $L$ to be a constant for which
\begin{equation}\label{eqn.LDef}
	\max\bigl(|(\lambda,\bth)|,\ |\Theta(\lambda,\bth)|\bigr)\ \le\ L\!\!\max_{2\le i\le d}\bigl|Y_i(\Theta(\lambda,\bth))\bigr|
\end{equation}
for all $(\lambda,\bth)$. Let $D=|\det(Y\circ\Theta)|$ and $J=2DV$, with $V$ the covolume of $\Gamma$.

The set
\[
	B=\Bigl\{(x_0,\lambda,\bth)\in\R\times\R^{r+z-1}\times\R^z:\ |x_0|\le\tfrac12\log T,\ \bigl|Y_i(\Theta(\lambda,\bth))\bigr|\le\eta_i\ \ (2\le i\le d)\Bigr\}
\]
is compact, convex and symmetric about the origin, and its volume is
\[
	\log T\cdot D^{-1}\prod_{i=2}^{d}2\eta_i=\frac{2^{d-1}\log T}{D}\prod_{i=2}^{d}\eta_i=2^{d}V,
\]
by the first condition in \eqref{eqn.BoxVol}. By Minkowski's convex body theorem $B$ contains a nonzero point of $\Gamma$, say $(\log v,\lambda(v),\bth)$ with $v\in P$. Replacing $v$ by $v^{-1}$ and $\bth$ by $-\bth$ if necessary, we may assume that $\log v\ge 0$. If $\log v=0$ then $v=1$ and $\lambda(v)=\bze$, while $|\bth|\le L\max_i\eta_i\le 1/4$ by \eqref{eqn.LDef} and the second condition in \eqref{eqn.BoxVol}; since $\bth\in 2\pi\Z^z$ this forces $\bth=\bze$ and the point to be the origin. Therefore $1<v\le T^{1/2}$.

Let $u=v^2$, so that $u\in\Zl^\times$ and $1<u\le T$. For $1\le k\le r$ we have that $\sigma_k(u)=\sigma_k(v)^2>0$, so $u$ is totally positive, and $\lambda_k(u)=2\lambda_k(v)$. For $1\le j\le z$ we have that $\Arg\sigma_{r+j}(u)\equiv 2\theta_j~\mathrm{mod}~2\pi$ with $|2\theta_j|\le 1/2<\pi$, so that $\Arg\sigma_{r+j}(u)=2\theta_j$. Since $\lambda_{r+z}$ is determined by the linear relation above, it follows that
\[
	\bm\zeta(u)=2\,\Theta\bigl(\lambda(v),\bth\bigr),
\]
and \eqref{eqn.BoxConc} now follows from the linearity of $Y$, together with \eqref{eqn.LDef}.
\end{proof}

With $Y=\Theta^{-1}$ and all of the $\eta_i$ equal, this is Lemma \ref{lem.Balance} with different constants; the generalization is only that the $d-1$ side lengths are now free, subject to their product.

\noindent\emph{Proof of Theorem \ref{thm.Weighted}:}\quad Apply Lemma \ref{lem.BalanceBox} with $Y=\Psi'|_H$, which is a linear isomorphism by Lemma \ref{lem.ErrorMap}, and write $L$ and $J$ for the resulting constants. Let $\|\Psi\|$ be a constant for which $|\Psi_i(\bm\xi)|\le\|\Psi\|\,|\bm\xi|$ for all $\bm\xi\in V_\R$ and all $i$, let $C'\ge 2$ be as in Section \ref{sec.Proof}, and set
\[
	f_{\mr{min}}=\min_{2\le i\le d}f_i(T),\qquad f_{\mr{max}}=\max_{2\le i\le d}f_i(T) .
\]
Suppose that $T\ge C'^2$, put
\[
	\mu=\left(\frac{J\log T}{\log (T/C')}\right)^{1/(d-1)},\qquad \eta_i=\frac{\mu}{f_i(T)}\quad(2\le i\le d),
\]
and note that $J^{1/(d-1)}\le\mu\le(2J)^{1/(d-1)}$, and that
\[
	\prod_{i=2}^{d}\eta_i=\frac{\mu^{d-1}}{f_2(T)\cdots f_d(T)}=\frac{\mu^{d-1}}{\log T}=\frac{J}{\log (T/C')},
\]
since the $f_i$ are admissible. This is the first condition in \eqref{eqn.BoxVol}, with $T$ replaced by $T/C'$.

Since $f_2(T)\cdots f_d(T)=\log T$ and there are $d-1$ factors, we have that $f_{\mr{max}}\ge(\log T)^{1/(d-1)}$, so hypothesis \eqref{eqn.WeightCond} gives
\begin{equation}\label{eqn.FminGrows}
	f_{\mr{min}}\ \ge\ \left(\frac{f_{\mr{max}}}{c}\right)^{1/2}\ \ge\ \frac{(\log T)^{1/(2(d-1))}}{c^{1/2}} .
\end{equation}
Assume that $T$ is large enough (depending on $\balph$ and $c$) so that $f_{\mr{min}}\ge 4L\mu$, which is the second condition in \eqref{eqn.BoxVol}, since $\max_i\eta_i=\mu/f_{\mr{min}}$.

Lemma \ref{lem.BalanceBox}, applied at $T/C'$, now produces a totally positive $u\in\Zl^\times$ with $1<u\le T/C'$,
\[
	\bigl|\Psi_i(\bm\zeta(u))\bigr|\le 2\eta_i\quad(2\le i\le d),\qquad
	|\bm\zeta(u)|\le\frac{2L\mu}{f_{\mr{min}}}\le\frac12 .
\]
Write $\bm\zeta=\bm\zeta(u)$. By \eqref{eqn.ZetaDef} and the estimates $|e^w-1-w|\le|w|^2$ and $|e^w-1|\le 2|w|$, valid for $|w|\le 1$, we have that $|\bm\epsilon-\bm\zeta|\le|\bm\zeta|^2$ and $|\bm\epsilon|\le 2|\bm\zeta|\le 1$. Therefore, for $2\le i\le d$,
\[
	\bigl|\Psi_i(\bm\epsilon)\bigr|\le\bigl|\Psi_i(\bm\zeta)\bigr|+\|\Psi\|\,|\bm\zeta|^{2}
	\le 2\eta_i+\frac{4\|\Psi\|L^2\mu^2}{f_{\mr{min}}^{2}} .
\]
Here we use hypothesis \eqref{eqn.WeightCond}, to bound the nonlinear part on the right hand side. It gives $f_{\mr{min}}^{-2}\le c\,f_{\mr{max}}^{-1}\le c\,f_i(T)^{-1}=c\,\eta_i/\mu$, so 
\[\frac{4\|\Psi\|L^2\mu^2}{f_{\mr{min}}^{2}}\le 4c\|\Psi\|L^2\mu\,\eta_i,\]
and hence $|\Psi_i(\bm\epsilon)|\le c'\eta_i$, for some constant $c'$ depending on $\balph$ and $c$.

Let $q$ and $\bp$ be as in Section \ref{sec.Proof}, for this $u$. By \eqref{eqn.ExactError},
\begin{align}\label{eqn.WeightedErrors}
	|p_1-q\alpha_1|&\le u^{-1/d}\bigl(1+\|\Psi\|\,|\bm\epsilon|\bigr)\le\bigl(1+\|\Psi\|\bigr)u^{-1/d},\nonumber\\
	|p_i-q\alpha_i|&\le\frac{c'\mu\,u^{-1/d}}{f_i(T)}\quad(2\le i\le d).
\end{align}
Assume now that $T$ is large enough so that $c'\mu/f_{\mr{min}}<1$. Then $q\not=0$: otherwise the integers $p_i=p_i-q\alpha_i$ satisfy $|p_i|<1$ for $2\le i\le d$ and so vanish, whence $\alpha_1^*u=p_1\alpha_1^*$ and $u=p_1$ is a rational integer which is a unit, contradicting $u>1$.

Let $n=|q|$, so that $1\le n\le C'u\le T$ by the first estimate of \eqref{eqn.KeyEstimate}, and let $p_i'=\mr{sgn}(q)p_i$, so that $\|n\alpha_i\|\le|p_i-q\alpha_i|$. Since $u^{-1/d}\le C'^{1/d}n^{-1/d}$, the bounds \eqref{eqn.WeightedErrors} are \eqref{eqn.Weighted}, with
\[
	C=C'^{1/d}\max\bigl(1+\|\Psi\|,\ c'(2J)^{1/(d-1)}\bigr).
\]
This completes the proof of Theorem \ref{thm.Weighted}.

\noindent\emph{Proof of Theorem \ref{thm.Biquad}:}\quad Here $d=3$, $r=3$ and $z=0$, so $V_\R=\R^3$ and $H$ is the plane $\zeta_1+\zeta_2+\zeta_3=0$. Label the nontrivial embeddings of $K=\Q(\sqrt a,\sqrt b)$ by
\[
	\sigma_1:(\sqrt a,\sqrt b)\mapsto(-\sqrt a,\sqrt b),\qquad
	\sigma_2:(\sqrt a,\sqrt b)\mapsto(\sqrt a,-\sqrt b),\qquad
	\sigma_3=\sigma_1\sigma_2 .
\]
Since the traces of $\sqrt a$, $\sqrt b$ and $\sqrt{ab}$ all vanish, we have that $\Tr(\alpha_i\alpha_\ell)=4\alpha_i^2\delta_{i\ell}$ for $1\le i,\ell\le 3$ and $\Tr(\alpha_0\alpha_\ell)=4\delta_{0\ell}$, so the dual basis is given by $\alpha_0^*=1/4$ and $\alpha_i^*=1/(4\alpha_i)$. In particular $\alpha_1^*=1/(4\sqrt{ab})$, and
\[
	\sigma_1(\alpha_1^*)=\sigma_2(\alpha_1^*)=-\alpha_1^*,\qquad \sigma_3(\alpha_1^*)=\alpha_1^* .
\]
Substituting these, together with $\sigma_3(\alpha_1)=\alpha_1,$ $\sigma_2(\alpha_2)=\alpha_2$, $\sigma_1(\alpha_3)=\alpha_3$ and $\sigma_k(\alpha_i)=-\alpha_i$ in the remaining cases, into \eqref{eqn.PsiTrace} gives
\begin{equation}\label{eqn.PsiBiquad}
	\Psi_1(\bm\xi)=\tfrac12(\xi_1+\xi_2),\qquad
	\Psi_2(\bm\xi)=2\alpha_1^*\alpha_2(\xi_1-\xi_3),\qquad
	\Psi_3(\bm\xi)=2\alpha_1^*\alpha_3(\xi_2-\xi_3).
\end{equation}
Accordingly we set $y_2(\bm\zeta)=\zeta_1-\zeta_3$ and $y_3(\bm\zeta)=\zeta_2-\zeta_3$. On $H$ we have that
\[
	\zeta_3=-\tfrac13(y_2+y_3),\qquad \zeta_1=\tfrac13(2y_2-y_3),\qquad \zeta_2=\tfrac13(2y_3-y_2),
\]
so $Y=(y_2,y_3)$ is a linear isomorphism from $H$ onto $\R^2$ with $|\bm\zeta|\le|Y(\bm\zeta)|$. Here $z=0$ and $\Theta(\lambda_1,\lambda_2)=(\lambda_1,\lambda_2,-\lambda_1-\lambda_2)$, so $|(\lambda_1,\lambda_2)|\le|\Theta(\lambda_1,\lambda_2)|$ and \eqref{eqn.LDef} holds with $L=1$.

The essential point is that, by \eqref{eqn.ZetaDef} and \eqref{eqn.PsiBiquad},
\begin{equation}\label{eqn.FlatError}
	\begin{split}
		&\Psi_2(\bm\epsilon)=2\alpha_1^*\alpha_2\bigl(e^{\zeta_1}-e^{\zeta_3}\bigr)=2\alpha_1^*\alpha_2\,e^{\zeta_3}\bigl(e^{y_2}-1\bigr),\quad\text{and}\\
	&\Psi_3(\bm\epsilon)=2\alpha_1^*\alpha_3\,e^{\zeta_3}\bigl(e^{y_3}-1\bigr).
	\end{split}
\end{equation}
The set of $\bm\zeta\in H$ on which $\Psi_i(\bm\epsilon)$ vanishes is the hyperplane $y_i=0$, an equation in only one of the two variables $y_2, y_3$. This circumvents the need to consider the second order approximation to $\bm\epsilon$, as in the proof of Theorem \ref{thm.Weighted}, and it is what allows us to dispense with the extra hypothesis \eqref{eqn.WeightCond}.

Apply Lemma \ref{lem.BalanceBox} with $Y=(y_2,y_3)$ and $L=1$, and write $J$ for the resulting constant. Given $T\ge C'^2$, with $C'$ as in Section \ref{sec.Proof}, let $i_1$ be an index in $\{2,3\}$ with $f_{i_1}(T)=\max(f_2(T),f_3(T))$, let $i_2$ be the other one, and put
\[
	J'=\frac{J\log T}{\log(T/C')},\qquad \mu_{i_1}=8J',\qquad\mu_{i_2}=\frac18,\qquad \eta_i=\frac{\mu_i}{f_i(T)} .
\]
Then $\mu_{i_1}\mu_{i_2}=J'$, so that $\eta_2\eta_3=J'/(f_2(T)f_3(T))=J/\log(T/C')$, which is the first condition in \eqref{eqn.BoxVol} at height $T/C'$. For the second, $\eta_{i_2}\le\mu_{i_2}<1/4$ always, while $f_{i_1}(T)\ge(\log T)^{1/2}$ because $f_2(T)f_3(T)=\log T$, so $\eta_{i_1}=8J'/f_{i_1}(T)\le 1/4$ for all $T\ge T_0$, with $T_0$ depending only on $a$ and $b$. Nothing is required of $f_{i_2}$, which may be identically $1$.

Lemma \ref{lem.BalanceBox} produces a totally positive $u\in\Zl^\times$ with $1<u\le T/C'$, $|y_i(\bm\zeta(u))|\le 2\eta_i$ and $|\bm\zeta(u)|\le 2\max_i\eta_i\le 1/2$. Writing $\bm\zeta=\bm\zeta(u)$, we have $e^{\zeta_3}\le e^{1/2}<2$ and $|y_i|\le 2|\bm\zeta|\le 1$, so that $|e^{y_i}-1|\le 2|y_i|$, and \eqref{eqn.FlatError} gives
\begin{equation}\label{eqn.BiquadErrors}
	\bigl|\Psi_i(\bm\epsilon)\bigr|\le 8\alpha_1^*\alpha_i\,|y_i|\le 16\alpha_1^*\alpha_i\,\eta_i\qquad(i=2,3),
\end{equation}
while $|\Psi_1(\bm\epsilon)|\le|\bm\epsilon|\le 2|\bm\zeta|\le 1$ by \eqref{eqn.PsiBiquad}.

Let $q$ and $\bp$ be as in Section \ref{sec.Proof}, for this $u$. By \eqref{eqn.ExactError}, \eqref{eqn.BiquadErrors} and $u>1$,
\[
\begin{gathered}
	|p_{i_2}-q\alpha_{i_2}|\le 16\alpha_1^*\alpha_{i_2}\mu_{i_2}=2\alpha_1^*\alpha_{i_2}<\tfrac12,\\
	|p_{i_1}-q\alpha_{i_1}|\le 16\alpha_1^*\alpha_{i_1}\eta_{i_1}=\frac{128\alpha_1^*\alpha_{i_1}J'}{f_{i_1}(T)},
\end{gathered}
\]
where the first bound uses $4\alpha_1^*\alpha_2=1/\sqrt b$ and $4\alpha_1^*\alpha_3=1/\sqrt a$, and the second is less than $1$ for all $T\ge T_0$, after enlarging $T_0$ if necessary, since $f_{i_1}(T)\ge(\log T)^{1/2}$. Therefore $q\not=0$, exactly as in the proof of Theorem \ref{thm.Weighted}. Putting $n=|q|$ we have $1\le n\le C'u\le T$ and $u^{-1/3}\le C'^{1/3}n^{-1/3}$, so \eqref{eqn.ExactError} and \eqref{eqn.BiquadErrors} give \eqref{eqn.BiquadConc} with
\[
	C=C'^{1/3}\max\bigl(2,\ 16\alpha_1^*\alpha_2M_0,\ 16\alpha_1^*\alpha_3M_0\bigr),\qquad M_0=\max(16J,\tfrac18),
\]
since $\mu_2,\mu_3\le M_0$ for every $T$ and every admissible pair, because $\mu_{i_1}=8J'\le 16J$. This completes the proof of Theorem \ref{thm.Biquad}.

{\footnotesize
\noindent
Department of Mathematics\\
University of Houston\\
Houston, TX, United States\\
KD: kkavita@cougarnet.uh.edu\\
JF: jbflynn@cougarnet.uh.edu\\
AH: haynes@math.uh.edu
}

\end{document}